\documentclass[11pt]{amsart}
\usepackage{amsmath,amssymb}
\usepackage{graphicx}
\usepackage{amsfonts,amscd,latexsym,epsfig,oldgerm,psfrag}

\newcommand{\labell}[1] {\label{#1}}

\newlength{\facewd} \newlength{\faceht}%

\newcommand{\ba} {{\bf a}}
\newcommand{\bm} {{\bf m}}

\newcommand{\bW} {{\bf W}}

\newcommand{\intt}{{\rm int\,}}
\newcommand{\Orb}{{\rm Orb\,}}

\newcommand{\less}{{\smallsetminus}}

\newcommand{\al}{{\alpha}}

\newcommand{\la}{{\lambda}}

\newcommand{\Cc}{{\mathcal C}}

\newcommand{\Ff}{{\mathcal F}}

\newcommand{\Ee}{{\mathcal E}}

\newcommand{\Ss}{{\mathcal S}}

\newcommand{\ov}{\overline}

\newcommand{\N}{{\mathbb N}}
\newcommand{\PP}{{\mathbb P}}

\newcommand{\R}{{\mathbb R}}
\newcommand{\C}{{\mathbb C}}

\newcommand{\Z}{{\mathbb Z}}

\newcommand{\Nn}{{\mathcal N}}
\newcommand{\Dd}{{\mathcal D}}

\newcommand{\SSS}{{\smallskip}}
\newcommand{\QED}{{\hfill $\Box$\MS}}
\newcommand{\se} {{\stackrel{s}\hookrightarrow}}

\newtheorem{theorem}{Theorem}[section]
\newtheorem{thm}[theorem]{Theorem}

\newtheorem{cor}[theorem]{Corollary}
\newtheorem{lemma}[theorem]{Lemma}

\newtheorem{prop}[theorem]{Proposition}

\newtheorem{example}[theorem]{Example}

\newtheorem{rmk}[theorem]{Remark}

\numberwithin{figure}{section}
\numberwithin{equation}{section}

\newcommand{\MS}{{\medskip}}

\newcommand{\NI}{{\noindent}}

\begin{document}

\title{The Hofer conjecture on embedding symplectic ellipsoids}
\author{Dusa McDuff} \thanks{partially supported by NSF grant DMS 0905191.}
\address{(D.~McDuff)
Department of Mathematics, 
Barnard College, Columbia University, New York, 
NY 10027-6598, USA.}
\email{dmcduff@barnard.edu}
\date{August 8, 2010, revised November 30, 2010}

\begin{abstract} 
In this note we  show that one open $4$-dimensional ellipsoid embeds symplectically into another if and only the ECH capacities of
 the first are no larger than those of the second.  This proves a conjecture due to Hofer.
 The argument uses the equivalence of the ellipsoidal embedding 
 problem with a ball embedding problem that was recently established by McDuff.  Its method is inspired by Hutchings' recent results  on embedded contact homology (ECH) capacities but does not use them. 
\end{abstract}

\maketitle
\section{Introduction}

Consider the ellipsoid $E(a,b): = \{z\in \C^2: \frac{|z_1|^2}a + \frac{|z_2|^2}b\le 1\}$ with the symplectic structure induced from the standard structure on Euclidean space.  Define $\Nn(a,b)$ to be the sequence of numbers  formed by arranging all the positive integer combinations $ma+nb, m,n\ge 0,$ in nondecreasing order (with repetitions). 
We say that  $\Nn(a,b)$ is less than or equal to $\Nn(c,d)$ (written $\Nn(a,b)  \preccurlyeq \Nn(c,d)$) if, for all $k\ge 0$,  the $k$th entry of $\Nn(a,b)$ is at most equal to the $k$th entry in $\Nn(c,d)$.
Hofer's conjecture  evolved as 
earlier guesses, such as those by Cieliebak, Hofer, Latschev and Schlenk in \cite{CHLS}, proved inadequate.  Finally, in private conversation, he
conjectured   that the numbers 
$\Nn(a,b)$ should detect precisely when these embeddings exist. 

We show in this note that this is indeed the case.

\begin{thm}\labell{thm:hof}  There is a symplectic embedding $\intt E(a,b)\;\se \; E(c,d)$ exactly if
$$
\Nn(a,b)  \preccurlyeq \Nn(c,d).
$$
\end{thm}

Using embedded contact homology (ECH), 
Hutchings showed in \cite{H} that the condition  $\Nn(a,b)  \preccurlyeq \Nn(c,d)$ 
is necessary.  
In fact, the main result of his paper is that there are quantities 
called ECH capacities, 
defined for any closed bounded subset of $\R^4$, 
that are monotone under symplectic embeddings.  The application to embedding ellipsoids
then follows because  
 the ECH capacities of $E(a,b)$ are just the sequence $\Nn(a,b)$. 
  Hutchings also shows that his ECH capacities give sharp obstructions 
  to the problem of embedding a union of disjoint balls into a ball. 
  
In the case of embedding an ellipsoid into a ball,
McDuff--Schlenk \cite[Thm.~1.1.3]{MS}  calculated exactly when 
the embedding exists, and concluded that in this case 
the condition  $\Nn(a,b)  \preccurlyeq \Nn(c,d)$ 
is sufficient.   Combining these results, we see
that Hofer's conjecture holds when the target is a ball.   
Below we prove the result in general by a much shorter argument
that uses none of the geometric results in ECH. 
Instead it uses some elementary combinatorics  that develop some of 
Hutchings'  ideas,
as well as the result from McDuff \cite{M} that reduces the 
ellipsoidal embedding problem to a ball embedding problem.
See Hutchings \cite{Hs} for a survey that gives more of the background.

The higher dimensional analog of Theorem \ref{thm:hof}  is completely open; there is even no good guess of what the answer should be. However the analog of the Hofer conjecture does not hold.  The first counterexamples are due to Guth \cite{Gu}  who showed that there are constants $a,b,c$ such that   $E(1,R,R)$ embeds symplectically in $ E(a,b,cR^2)$ for all $R>0$, with similar results in higher dimensions.   In \cite{HK} Hind--Kerman improved Guth's embedding method to
show that $E(1,R,\dots, R)$ embeds in $E(a,a,R^2,\dots,R^2)$
whenever $a> 3$, but found an obstruction when $a< 3$.

\begin{example}\labell{ex:1}\rm (i)
The sequence $\Nn(a,a)$ is
$$
\Nn(a,a) = (0,a,a,2a,2a,2a,3a,3a,3a,3a,\dots).
$$
In other words, for each $d$ there are $d+1$ entries of $da$ 
occurring as the terms $\Nn_k(a,a)$ for $\frac 12(d^2 + d) \le k \le \frac 12(d^2 + 3d)$.  Thus $\Nn(a,a)$ is the maximal sequence with $\Nn_k = da$ for $k = \frac 12(d^2 + 3d)$.
\MS

\NI (ii) When $k = \frac 12(d^2+3d)$, the sequence
$$
 \Nn(1,4) = (0,1,2,3,4,4,5,5,6,6,7,7,8,8,8,\dots)
$$
has $\Nn_k(1,4) = 2d = \Nn_k(2,2)$. Hence the maximality of
$\Nn(2,2)$ implies that
 $\Nn(1,4)\preccurlyeq \Nn(2,2)$.  Thus $\intt E(1,4)\se B(2)$. (Here, and elsewhere, we write $X\se Y$ to mean that $X$ embeds symplectically in $Y$.)  
 
 The first construction of  an embedding of this kind is due to Opshtein \cite{Op}.\footnote{
 In fact he constructed an explicit  embedding from $\intt E(1,4)$ into projective space by using properties of neighborhoods of curves of degree $2$, but one can easily arrange that the embedding  avoids a line so that  there is a corresponding embedding into a ball. Cf. also Theorem 4 in \cite{Op2}.}  
 The paper \cite{M} develops a general  method of embedding ellipsoids, which in most cases is not very geometric.  However, 
 as is shown in \cite[\S1]{M}, in the special case of $E(1,4)$ the argument can be made  rather explicit.  One still cannot see the geometry of the  image as clearly as in Opshtein because one uses symplectic inflation to increase the size of the image of the ellipsoid, i.e. rather than embedding larger and larger ellipsoids into a fixed ball, one embeds a small ellipsoid  $E$ into the ball $B$ and then increases the relative size  of $E$ by distorting  the symplectic form on $B\less E$.
\end{example}

\begin{rmk}\rm  (i)
We phrase all our  results in terms of  embedding the interior of $X$ into $Y$ (or, equivalently, embedding $\intt X$ into $\intt Y$), while Hutchings talks about embedding $X$ into the interior of $Y$. But these amount to the same  when $X$ is the disjoint union of ellipsoids, since in this case there is a symplectic embedding $\intt X \se Y$ exactly if  $\la X$ embeds symplectically into $\intt Y$ for all $\la<1$; cf. \cite[Cor.~1.5]{M}.\MS

\NI (ii) The current 
methods extend to give a simple numerical criterion for embedding  disjoint unions of ellipsoids  into an ellipsoid. See Proposition \ref{prop:many} for a precise result.
Also, all the methods used here extend to the case when the target manifold is a polydisc,
i.e. a product of
two discs with a product form, cf. M\"uller \cite{DoM}, or a  blowup of a rational or ruled surface.
 However, just as in \cite{Mdef}, one gets no information when the target is a closed $4$-manifold with $b_2^+>1$ such as $T^4$ or a surface of general type.
\end{rmk}

See Bauer \cite{Bau} for more information on the numerical properties  of the sequences $\Nn(a,b)$.
\MS

\NI {\bf Acknowledgements}  I warmly thank Michael Hutchings for his patient explanations of the index calculations in ECH and for many illuminating discussions, Felix Schlenk for some useful comments on an earlier version of this note, and also FIM at ETH, Z\"urich, for 
providing a very stimulating environment at Edifest, 2010.

\section{Combinatorics}

A standard continuity argument implies that it suffices to prove
Theorem \ref{thm:hof} when the ratios $b/a$ and $d/c$ are rational.
One of the
main results of \cite{M} states 
that for each integral ellipsoid $E(a,b)$ 
there is a sequence of integers $\bW(a,b): = (W_1,\dots,W_M)$ called the (normalized) {\bf weight sequence} of $a,b$, such that $\intt E(a,b)$ embeds in  $B(\mu)$ exactly if the disjoint union $\sqcup \,\intt B(\bW): = \sqcup \,\intt B(W_i)$ embeds in $B(\mu)$.  This section  shows  how $\Nn(a,b)$ may be calculated in terms of $\bW(a,b)$.

We begin with some definitions.  They are basically taken from \cite{M}, 
but are modified as in \cite{MS}. 
Given positive  integers $p,q$ with $q\le p$ we denote by $\bW(p,q) = \bW(q,p)$
the normalized weight sequence of $p/q$.   
Thus $\bW(p,q) = 
 (W_1,W_2,\dots,W_M)$  is a finite sequence of positive integers  defined recursively by the following rules:\footnote
 {
The proof that the weight expansion $\bW(p,q)$ described above agrees with that used in \cite{M} is given in the appendix to \cite{MS}. }
\MS

$\bullet$   
$W_1 = q$ and $ W_n \ge W_{n+1}>0$ for all $n$;

\SSS 
$\bullet$  
if $ W_i>W_{i+1} = \dots = W_{n}$ (where we set $W_0 := p$), then
$$
W_{n+1} = \min \bigl\{W_n, \; W_i - (n-i) W_{i+1} \bigr\};
$$

$\bullet$ 
the sequence stops at $W_M$ if the above formula gives $W_{M+1}=0$.\MS

\begin{figure}[htbp] 
   \centering
   \includegraphics[width=2in]{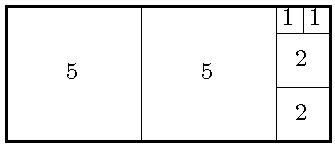} 
   \caption{One obtains the weights by cutting  a rectangle into squares:
   $W(5,12) = (5,5,2,2,1,1)=:(5^{\times 2},2^{\times 2},1^{\times 2})$.}
   \label{fig:3}
\end{figure}
\NI

It is often convenient to write $\bW(p,q)$ as 
\begin{equation}\labell{eq:W}
\bW(p,q)= \bigl(X_0^{\times\ell_0}, X_1^{\times \ell_1},\dots, X_K^{\times \ell_K}\bigr),
\end{equation}
where $X_0> X_1>\dots >X_K>0$ and $\ell_K\ge 2$.  Thus the $\ell_i$ are the multiplicities of
the entries in $\bW(p,q)$ and, as is well known, give the continued fraction expansion of $
p/q$: namely
\begin{equation}\labell{eq:ell}
\frac pq = \ell_0 + \tfrac 1{\ell_1 + \tfrac 1{\ell_2 + \dots \tfrac 1{\ell_K}}} =: [\ell_0;\ell_1,\dots,\ell_K].
\end{equation}
In this notation, the defining formulas for the terms in $\bW(p,q)$ become:
$$
X_{-1}: = p, \;\;X_0 = q, \;\;
X_{i+1} = X_{i-1}-\ell_i X_i,\; i\ge 0.
$$
In particular, because $X_1 = p-\ell_0q$,
\begin{equation}\labell{eq:Wpq}
W(p,q) = \bigl(q^{\times\ell_0}, X_1^{\times \ell_1}, X_2^{\times \ell_2},\dots\bigr) = 
 \bigl(q^{\times\ell_0}, W(q,X_1)\bigr).
\end{equation}
 More generally,  the following holds.
 
 \begin{lemma}\labell{le:W} Define $p,q, X_i, \ell_i, 0\le i\le K$ as in equations \eqref{eq:W} and \eqref{eq:ell} and set $X_{-1}: = p, \ell_{-1}: = 1$.  Then for all $i=0,\dots,K-1$, we have
 $$
  \bW(X_{i-1}, X_i) = \bigl(X_i^{\times \ell_i}, \bW(X_i, X_{i+1})\bigr).
  $$
  \end{lemma}
  
  \begin{rmk}\rm  
  The weight sequence $\bW(p,q)$ does not seem to be mentioned in elementary treatments of continued fractions. Instead, one considers
  the convergents $p_k/q_k: = [\ell_0;\ell_1,\dots,\ell_k]$ of $p/q$. However, 
it is well known that the two mirror  fractions 
  $p/q: = [\ell_0;\ell_1,\dots,\ell_K]$ and $P/Q: = [\ell_K;\ell_{K-1},\dots,\ell_0]$
  have the same numerator $p=P$. 
  It follows that when $p,q$ are relatively prime the sequence
  $$
  (X_0,X_1,\dots,X_K),
  $$
  when taken in reverse order, is just
  the sequence of numerators of the convergents  of the mirror $P/Q$.
  More precisely, $X_{K-k} =P_{k-1}$ for $1\le k \le K+1$.  Further,  the rectangle definition of the $W_i$ 
  (cf. Figure \ref{fig:3}) easily implies
  that $\sum_i W_i^2 = pq$.  Another, less obvious, quadratic relation for the $W_i$ is discussed in \cite[\S2.2]{MS}.
  \end{rmk}
  
 We now show that $\Nn(a,b)$ may be calculated 
 from the weights $\bW(a,b)$ using an operation $\#$ first considered by Hutchings.
 Given two nondecreasing 
 sequences $\Cc=(\Cc_k)_{k\ge 0}, \Dd=(\Dd_k)_{k\ge 0}$ with $\Cc_0=\Dd_0 = 0$,
 define $\Cc\# \Dd$ by
 $$
 (\Cc\#\Dd)_k: = \max_{0\le i\le k} (\Cc_i + \Dd_{k-i}).
 $$
 
 The next result follows immediately from the definition.
 
 \begin{lemma}\labell{le:assoc}  \begin{itemize}\item[(i)] The operation $\Cc,\Dd\mapsto \Cc\#\Dd$ is  associative and commutative.  
 \item[(ii)] If $\Cc\preccurlyeq \Cc'$ then $\Cc\#\Dd \preccurlyeq \Cc'\#\Dd$  for all sequences $\Dd$.
 \end{itemize}
 \end{lemma}
 
  For any sequence $\ba=(a_1, \dots,a_M)$ of positive integers, we define
  $$
  \Nn(\ba): = \Nn(a_1,a_1)\# \Nn(a_2,a_2)\#\dots \#\Nn(a_M,a_M).
  $$
  If $a: = a_1=\dots = a_M$ we abbreviate this product as $\#^M \Nn(a,a)$.
  
 To understand the effect of the operation $\#$  on sequences of the form $\Nn(a,b)$,
 it is convenient to interpret the numbers $\Nn_k(a,b)$ in terms of lattice counting as in  Hutchings \cite[\S3.3]{H}.  
For $A>0$ consider the triangle 
$$
T_{a,b}^A: =\bigl \{(x,y)\in\R^2 : x,y\ge 0 ,\, ax+by\le A\bigr\}.
$$  
Each integer point $(m,n)\in T_{a,b}^A$ gives rise to an element of the sequence $\Nn(a,b)$ that is $\le A$. If
$a/b$ is irrational  there is  for all $A$  at most
  one integer point on the slant edge of $T_{a,b}^A$. It follows that
 $
 \Nn_k(a,b) = A$, where $A$ is such that
  $|T_{a,b}^A\cap\Z^2|= k+1$.
 Since for rational $a/b$ there might be more than one integral point on this slant edge, 
the general definition is:
 $$
 \Nn_k(a,b) = \inf\bigl\{A: |T_{a,b}^A\cap\Z^2|\ge k+1\bigr\}.
 $$

 \begin{figure}[htbp] 
    \centering
    \includegraphics[width=5in]{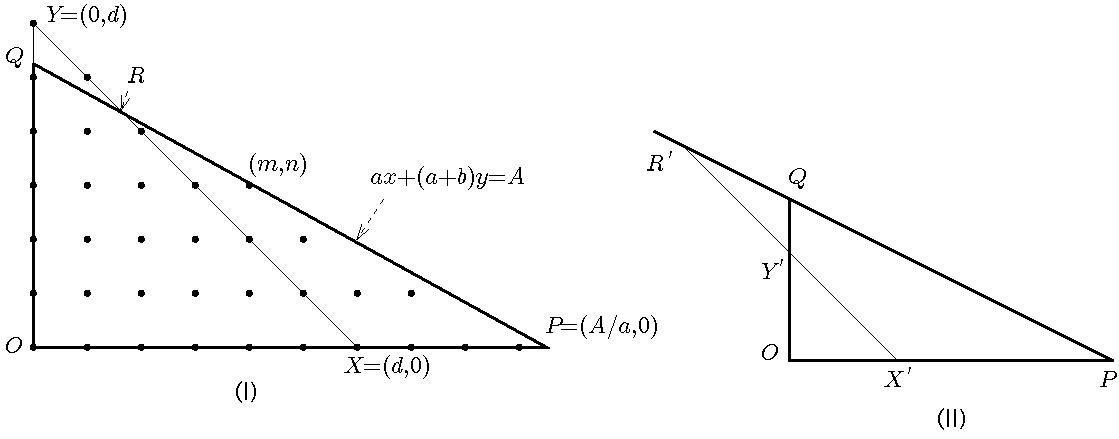} 
    \caption{Different cuts in Lemma \ref{le:hof1}.}
    \labell{fig:1}
 \end{figure}
\begin{lemma}\labell{le:hof1} For all $a,b> 0$, we have $\Nn(a,a) \#\Nn(a,b) =  \Nn(a,a+b)$.
More generally, for all $\ell\ge 1$, we have 
$\bigl(\#^\ell \Nn(a,a) \bigr)\#\Nn(a,b) =  \Nn(a,b + \ell a)$.
\end{lemma}
\begin{proof}  By continuity and scaling, it suffices to prove this when $a,b\in \Z$.  Suppose that $\Nn_k(a,a+b) = A$.  Then there is at least one integer point $(m,n)$ on the slant edge $QP$ of the triangle $T: = T_{a,a+b}^A$; see Fig \ref{fig:1} (I). 
Let $d: = \lceil\frac A{a+b} \rceil$ be the smallest  integer greater than $\frac A{a+b}$, and 
let $R$ be the point where the line joining $Y=(0,d)$ to $X=(d,0)$ meets the slant edge of $T$.  Then because the intersection of the line $x+y=d-1$ with the first quadrant
lies entirely in $T$ all the points in the interior of triangle $OXY$ lie in $T$.  Hence 
$\Ss: = T\cap \Z^2$ divides into two sets $\Ss_1$ and $\Ss_2$, where $\Ss_1$ contains $X$ plus all points $(m,n)$ in $T$ with 
$m+n<d$, 
and $\Ss_2$ consists of all other points in $\Ss$.  
Let 
$$
|\Ss_1|= k_1+1,\quad |\Ss_2|= k_2.
$$
Then $k_1+k_2=k$, and our remark above about the triangle $OXY$ implies
$\Nn_{k_1}(a,a) = da$.

Let $\al$ be   the  integral affine transformation that fixes $Y$  and translates the $x$ axis by $-d$ so that  $X$ goes  to the origin $O$. Then $\al$ takes the
 triangle $XRP$ to the triangle $T_{a,b}^B$, where $B/a = A/a - d$.
 The set of integral points in $T_{a,b}^B$ is $\al(\Ss_2\cup \{X\})$.
 Hence $\Nn_{k_2}(a,b)= B=A-da$.
 Thus  $\Nn_{k_1} (a,a) + \Nn_{k_2}(a,b) = \Nn_k(a,a+b)$.
 
 We claim that for all other $i\le k$ we have $\Nn_{i} (a,a) + \Nn_{k-i}(a,b) \le \Nn_k(a,a+b)$.
To see this, we slightly modify the above argument as follows.  
If $i> k_1$ then $\Nn_i(a,a) = ad'$ for some $d'\ge d$.    
Decompose $T$ by the line $x+y=d'$ as above. Then  choose  $\Ss_1'\subset \Ss$ to contain $X'=(d',0)$ together with all points in $T$ with $x+y<d'$
and let $\Ss_2'= \Ss\less \Ss_1'$.  Then $k_1'+1: =|\Ss_1'| \le i+1$  (since now there may be some
integer  points in the interior of triangle $QY'R'$) so that $\Nn_{k_1'}(a,a) \le ad'$, while, as above, $\Nn_{k-k_1'}(a,b)= A-ad'$.   

If $i<k_1$ then we choose $d'$ as before and again slice $T$ by the line $x+y=d'$.  The corresponding triangle $OX'Y'$ is illustrated in Figure \ref{fig:1} (II). The line $X'Y'$ now meets the slant edge of $T$ at $R'$ lying beyond $Q$.  
Hence, if we partition the integral points in $T$ as before, $k_2'+1: = |S_2|+1$ 
is at most the number $\ell+1$ of integral points in $X'R'P$ (and may well be strictly smaller).  Thus 
$\Nn_{k_2'}(a,b) \le \Nn_{\ell}(a,b) = A-ad'$. On the other hand, $\Nn_{k_1'}(a,a) = ad'$ as before.

This completes the proof of the first statement.  The second follows immediately by induction.
\end{proof}

\begin{cor}\labell{cor:1} 
Let $\bW(a,b)$ be the  weight sequence for $(a,b)\in \N^2$.  
Then $\Nn(\bW(a,b))$ $= \Nn(a,b)$. 
\end{cor}
\begin{proof} Without loss of generality, we may suppose that $a,b$ are relatively prime and that $a\le b$.   
If $a=1$ then $
\Nn(1,b) = \#^b \Nn(1,1)
$
by the second statement in Lemma \ref{le:hof1}.

For a general pair $(a,b)$ we argue by induction on the length $K$ of the weight expansion
 $\bW(a,b) = \bigl(X_0^{\times \ell_0},\dots,X_K^{\ell_K}\bigr).
 $
We saw in \eqref{eq:Wpq} above that  $X_1=b-\ell_0 a$ and that
 $\bW(X_1, a) = \bigl(X_1^{\times \ell_1},\dots,X_K^{\ell_K}\bigr).
 $  Therefore, we may assume by induction that 
 $\Nn\bigl(\bW(X_1,a)\bigr)= \Nn(X_1,a)$.  Thus
 \begin{eqnarray*}
 \Nn\bigl(\bW(a,b)\bigr) &=&  \bigl(\#^{\ell_0}\Nn(a,a)\bigr)
 \#\Nn\bigl(\bW(X_1,a)\bigr)\\
 &=&  \bigl(\#^{\ell_0}\Nn(a,a)\bigr)\#
 \Nn(X_1,a) \\
 &=&\Nn(a, X_1+\ell_0 a) = \Nn(a,b),
 \end{eqnarray*}
where the first equality follows from the definition since 
$\bW(a,b)) = \bigl(a^{\times \ell_0}, \bW(X_1,a)\bigr)$, 
the second holds by the inductive hypothesis, and the third by   
Lemma \ref{le:hof1}.
\end{proof}

  The first part of the next lemma was independently observed by  David Bauer during Edifest 2010.
  \begin{figure}[htbp] 
     \centering
     \includegraphics[width=5in]{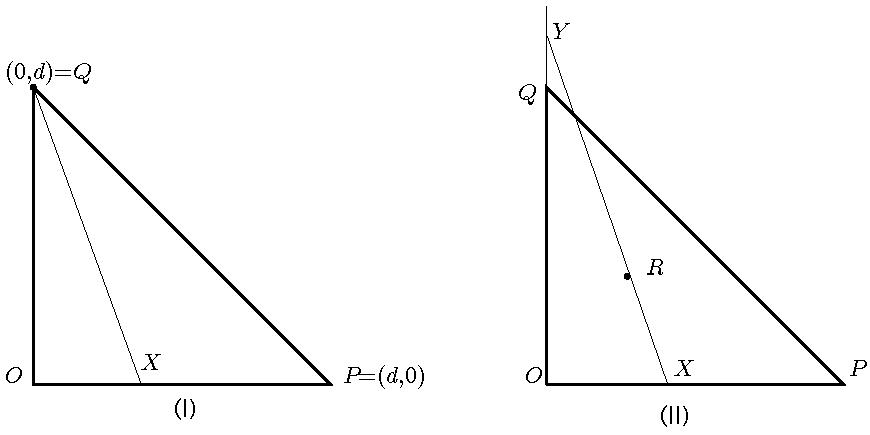} 
     \caption{Decompositions considered in Lemma \ref{le:hof2}.}
     \label{fig:2}
  \end{figure}
  
  \begin{lemma}\labell{le:hof2}  
  \begin{itemize}\item[(i)] Given integers $0<a<b$ we have $\Nn(a,b)\#\Nn(b-a,b) \le \Nn(b,b)$.
  \item[(ii)]  For each $k\ge 1$ there is $\ell$ such that  $\Nn_{k+\ell}(b,b) = \Nn_k(a,b) + 
  \Nn_\ell(b-a,b)$.
  \end{itemize}\end{lemma}
  \begin{proof} (i) Since $\Nn_i(b,b)$ is constant and equal to $bd$ on the sets of the form  $(d^2+d)/2\le i\le(d^2+3d)/2$, it suffices to consider the case $i = (d^2+3d)/2$.  Then the points counted by $\Nn_i(b,b)$ are those in the triangle $T= OPQ$ in diagram (I) in Figure \ref{fig:2}. Cut this triangle by the 
  line $QX$ of slope $-b/a$, and divide the integer points in $T$ into two groups $\Ss_1,\Ss_2$
  where $\Ss_1$ consists of $Q$ plus the points to the left of $QX$ and $\Ss_2$ is the rest.
  Let $|\Ss_1|=k_1+1$ and $ |\Ss_2|=k_2 $ as before.
  Then, counting the points in triangle $OQX = T_{b,a}^{ad}$,  
  we see that $\Nn_{k_1} = ad$.
  Further, if we move the triangle $PXQ$ first by the affine transformation that fixes the $x$-axis and takes $Q$ horizontally to the point $(d,d)$ and then by the reflection in 
  the vertical line $x=d$, it is a horizontal translate of $T_{b,b-a}^{d(b-a)}$.
  Hence counting the points in the triangle $PQX$,  
  we see that $\Nn_{k_2} = (b-a)d$.  Thus $\Nn_{k_1}+\Nn_{k_2} = bd$. This proves (i).
  
  To prove (ii), suppose $\Nn_k(a,b) = B$ and let $R$ be an integer point on the slant edge of the corresponding triangle $T_{b,a}^B$.  Let $Y$ be the point where the
  slant edge of this triangle meets the $y$-axis, so that $Y = (0,B/a)$, and let $X$ be where it meets the $x$-axis, so that $X = (B/b,0)$.
  Next,  let $Q = (0,d)$ be the integer point with $d=\lfloor B/a\rfloor$, and put $P=(d,0)$ as before.  Then  no points  
  in the triangle $T_{b,a}^B$ lie above the line $PQ$ since $|YQ|<1$.  We now divide
  the points in triangle $OPQ$ into two sets as before, with $\Ss_1$ the union of $R$ with all points to the left of $XY$ and $\Ss_2$ the rest.   
  Then $|\Ss_1|=k+1$ by construction.   Further, as in (i), the points in $\Ss_2$ lie in a triangle that is affine equivalent to $T_{b,b-a}^{A}$, where $A/b =|PX| = d-B/b$.   Moreover $\Ss_2$ contains $P$ (which corresponds to the origin) but not $R$, which is a point on the slant edge.
  Thus  $T_{b,b-a}^{A}$ has $|\Ss_2|+1$ integer points, so that, if $\ell: = |\Ss_2|$ we have $\Nn_\ell(b,b-a) = A=db-B$.  Therefore $\Nn_k(a,b) + \Nn_\ell(b,b-a) = db$.
  But $k+\ell = (d^2+3d)/2$ by construction, so that $\Nn_{k+\ell}(b,b) = db$.
  The result follows.
   \end{proof}
  
  \begin{cor}\labell{cor:2} Let $\Cc$ be a nondecreasing sequence of nonnegative numbers
  such  that $\Cc\#\Nn(d-c,d)\le \Nn(d,d)$.  Then
  $\Cc\le \Nn(c,d)$.
  \end{cor}
  \begin{proof}  If not, there is $k$ such that $\Cc_k> \Nn_k(c,d)$.   
But by the previous lemma, there is $\ell$ such that  $\Nn_{k+\ell}(d,d) 
= \Nn_k(c,d) + 
  \Nn_\ell(d-c,d)$.  Then $\Cc_k+\Nn_\ell(d-c,d) > \Nn_{k+\ell}(d,d)$ 
contradicting the hypothesis.
  \end{proof}

\section{Proof of Theorem \ref{thm:hof}}

Our argument is based on the following key results.

\begin{prop}\labell{prop:M} \cite[Thm.~3.11]{M}
Let $c,d,e,f$ be any positive integers and $\la>0$.  Then there is a
symplectic embedding 
$$
\Phi_E:\;\; \intt \la  E(e,f)\;\;\se \;\;\intt E(c,d)
$$
 if and only if there is a
symplectic embedding 
$$
\Phi_B:\;\;  \intt \la B\bigl(\bW(e,f)\bigr)\; \sqcup\; \intt B\bigl(\bW(d-c, d)\bigr)\;\;\se\;\; B(d).
$$
\end{prop}

The proof that the numerical condition is sufficient  for an embedding to exist
involves a significant use of Taubes--Seiberg--Witten theory in conjunction with
the theory of $J$-holomorphic curves.  However it is easy to see why
it is necessary, since the ellipsoid  $E(p,q)$ decomposes into a union of balls
 whose sizes are given by the weights $\bW(p,q)$; cf Figure \ref{fig:4}.
 To see this, recall that the
 moment (or toric) image of the ball is affine equivalent to a standard triangle (a right-angled isosceles triangle), while that of an ellipsoid is an arbitrary right-angled triangle.  As the diagram shows,  the 
decomposition  of a rectangle into squares given by the weights (as in Figure \ref{fig:3}) yields a
corresponding decomposition of the rectangle into (affine) standard triangles; see  
\cite[\S2]{Mcf} for more detail.
 
\begin{figure}[htbp] 
   \centering
   \includegraphics[width=4in]{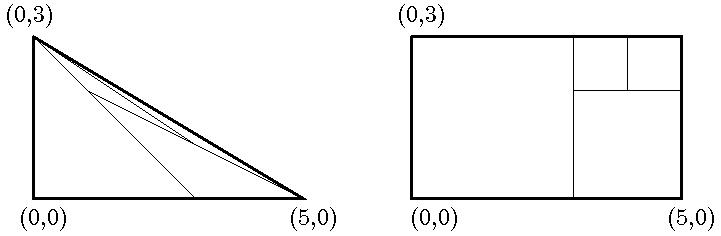} 
   \caption{Cutting a triangle  into standard triangles.}
   \label{fig:4}
\end{figure}

The next result follows from Hutchings' observation in 
\cite[Remark~1.10]{H} that the ECH capacities $\Nn(a,b)$ give a sharp obstruction for ball embeddings.  We give the proof to make it clear that it does not 
use any knowledge of ECH although it does use easier gauge theoretic  results.

\begin{prop}\labell{prop:H} \cite{H}
Let $\ba: = (a_1,\dots, a_M)$ be any sequence of positive numbers.  Then
there is a symplectic embedding  $\Phi_B:\sqcup_i \intt B(a_i)\se B(\mu)$ exactly if 
$\Nn(\ba): = \#_i \Nn(a_i,a_i)\preccurlyeq \Nn(\mu,\mu)$.
\end{prop}
 \begin{proof}   Hutchings showed in \cite[Proposition~1.9]{H}
 that $\Nn(\ba)\preccurlyeq \Nn(\mu,\mu)$ exactly if 
 \begin{equation}\labell{eq:ba}
 \sum_i m_ia_i\le \mu d, \quad \mbox{ whenever }
 \sum_i m_i^2 + m_i \le d^2 + 3d,
 \end{equation}
 where $(d;m_1,\dots,m_M)=: (d;\bm)$ is any sequence of nonnegative integers.
 This holds because
 the $k$th entry in the sequence 
 $\Nn(\ba)= \#_{i=1}^M \Nn(a_i,a_i)$ is the maximum of the numbers
 $$
 \sum_i\Nn_{k_i}(a_i,a_i) = \sum_i m_i a_i,
 $$ 
 where $\sum_i k_i = k$
 and $\frac 12 (m_i^2 + m_i) \le k_i\le \frac 12 (m_i^2 + 3m_i).$

On the other hand  it was shown by McDuff--Polterovich \cite{MP} that $\Phi_B$ exists exactly if for each $\la \in (0,1)$ there is a symplectic representative of the cohomology class $\al_\la: = \mu\ell -\la^2 \sum_ia_i e_i$ on the $M$-fold blowup 
$X_M: = \C P^2 \# M\ov{\C P}\,\!^2$ that is deformation equivalent to the small blowup of a form on $\C P^2$ and hence has standard first Chern class.  (Here $\ell$ and $e_i, i=1,\dots,M,$ denote
 the Poincar\'e duals of the classes of the line $L$ and the exceptional divisors $E_i$.)
Thus we need $\al_\la$ to lie
 in the symplectic cone $\Cc_K$ of $X_M$ given by the classes of all symplectic forms with first Chern class Poincar\'e dual to $-K: = 3L-\sum_iE_i$.  
 
 After preliminary work by McDuff \cite{Mdef} and Biran \cite{B} concerning the closure of $\Cc_K$,
 Li--Liu in \cite[Theorem~3]{LL}  used 
 Taubes--Seiberg--Witten theory to describe $\Cc_K$ 
  in the following terms.
 Let $\Ee_K\subset H_2(X_M;\Z)$ be the set of classes $E$ with 
 $-K\cdot E = 1$  that can be represented by smoothly 
 embedded spheres of self-intersection $-1$.  Then  Li--Liu showed that $\Cc_K$ is  connected and has the following description:
$$
\Cc_K = \{\al\in H^2(X_M;\R): \al^2>0, \langle 
\al,E\rangle>0\ \forall E\in \Ee_K\}.
$$
Notice that if $E = dL-\sum m_iE_i \in \Ee_K$ then
$\sum m_i^2 = d^2 + 1$ while $\sum m_i = 3d-1$.   Therefore each such $E$ does give rise to an inequality of the type considered in
\eqref{eq:ba}.\footnote
{
The cone $\Cc_K$ is described by strict inequalities, but when we let $\la\to 1$ these correspond to the $\le$ signs in \eqref{eq:ba}.} 
It is also easy to check that the inequalities in 
 \eqref{eq:ba} for $d\to \infty$ imply that $\sum a_i^2\le \mu$, which corresponds to 
 the volume condition $\al^2>0$; cf. \cite[Remark~3.13]{H}.
However, because  many tuples $(d;\bm)$ with 
$\sum_i m_i^2 + m_i \le d^2 + 3d$ do not correspond to elements in $\Ee_K$, it seems on the face of it that the conditions in
\eqref{eq:ba} are more stringent than the geometric condition $\al\in \Cc_K$. 
Lemma \ref{le:alg} below gives a purely algebraic argument 
 showing that this is not the case.
 \end{proof}

Let $\Ff$ be the set of all tuples $(d;\bm): = (d;m_1,\dots,m_M)$
such that $\sum_i(m_i^2 + m_i) \le d^2 + 3d$. 
We denote by $\Ff^+$ those elements $(d;\bm)\in \Ff$ with $d>0$ and $m_i\ge 0$ for all $i$.
Further we define $(\mu;\ba)\cdot (d;\bm): = d\mu -\sum a_i m_i=: d\mu-\ba\cdot \bm.$  Then, if $-K=(3;1,\dots,1)$, we have
$$ 
(d;\bm)\in \Ff\Longleftrightarrow (d;\bm)\cdot\big((d;\bm) - K\bigr)  \ge 0.
$$
Similarly, we  identify the cone $\Cc_K\subset H^2(X_M;\R)$ 
with the set of tuples $(\mu;\ba)$
given by the coefficients of the classes 
 $\al=\mu \ell -\sum a_ie_i\in \Cc_K$. In this notation, it suffices 
 to prove the following lemma.

\begin{lemma}\labell{le:alg}
Let $(\mu;\ba): = (\mu; a_1,\dots,a_M)$ be a tuple of nonnegative numbers such that 
\begin{itemize}\item[(i)]
$(\mu;\ba)\cdot (d;\bm)\ge 0$ for all $(d;\bm)\in \Ee_K$.
\item[(ii)] $\|\ba\|: = \sqrt {\sum_i a_i^2}\le \mu$.
\end{itemize}
Then $(\mu;\ba)\cdot (d;\bm)\ge 0$ for all  $(d;\bm)\in \Ff^+.$ 
\end{lemma}

The proof is based on the following elementary result, that is part of
 Li--Li  \cite[Lemma~3.4]{LL2}.
We say that a tuple
 $(\mu;\ba)$ is {\bf positive} if $\mu\ge 0$ and $a_i\ge 0$ for all $i$; that it is
 {\bf ordered} if $a_1\ge\dots\ge a_M$;
 and that a positive, ordered $(\mu;\ba)$  is {\bf reduced} if $\mu\ge a_1+a_2+a_3$.

\begin{lemma}\labell{le:red}  Let $(\mu;\ba)$ be reduced and $(d;\bm)$ be a positive tuple such that
$-K\cdot (d;\bm) = 3d-\sum m_i \ge 0$.  Then $(\mu;\ba)\cdot(d;\bm) \ge 0$.
\end{lemma}
\begin{proof}
  Because $\sum_im_i\le 3d$,  we may
 partition the list 
 $$
 (1^{\times m_1}, 2^{\times m_2},\dots, M^{\times m_M})
 $$
  (considered with multiplicities) into sets $I_n, 1\le n \le d,$
 where
 each $I_n = \{j_{n1},j_{n2},j_{n3}\}$ is a set of at most three distinct numbers chosen so that each element $j\in \{1,\dots, M\}$
 occurs in precisely $m_j$ different sets $I_n$.   
 Then 
 $$
\sum_{i=1}^M a_im_i = \sum_{n=1}^{d}\bigl(\sum_{i\in I_n}a_i\bigr).
$$
Further $\sum_{i\in I_n}a_i\le \mu$ for all $I_n$ because $(\mu;\ba)$ is reduced.  Hence
 $$
\sum_{i=1}^M a_im_i \le \sum_{n=1}^{d} \mu = d\mu,
 $$
as required.
\end{proof}
 
\NI{\bf Proof of Lemma \ref{le:alg}.} In the argument below we assume $M\ge 3$.
Since we allow the $a_i$ to be $0$, we can always reduce to this case by increasing $M$ if necessary.
  Next observe that it suffices to prove the result for integral tuples
$(\mu;\ba)$ and  $(d;\bm)$.  We suppose throughout 
that $(\mu;\ba)$ satisfies conditions (i) and (ii).
Further, if  
 $(d;\bm)$ is such that $\|\bm\| : = \sqrt{\sum m_i^2}\le d$,
then $(d;\bm)\cdot(\mu;\ba) \ge d\mu-\|\bm\|\,\|\ba\| \ge0 $ as required. Therefore  
we only need consider $(d;\bm)$ with $(d;\bm)\cdot (d;\bm)< 0$.

 Following  Li--Li \cite{LL2}, consider the Cremona transformation $Cr$ that acts on tuples by
$
Cr(d;\bm) = (d';\bm')$, where $m_j'=m_j$ for $j\ge 4$ and 
$$
d' = 2d-(m_1+m_2+m_3), \quad m_i' = d-(m_j+m_k) \mbox{ for } \{i,j,k\}=\{1,2,3\}.
$$
Then $Cr$ preserves the class $K$ and the intersection product,
 and hence preserves  $\Ff$.   Because $Cr$, when considered as acting on 
 $H_2(X_M)$,  is induced by a diffeomorphism (the reflection in the sphere in
 class $L-E_1-E_2-E_3$), it preserves the set of classes represented by embedded spheres and hence preserves $\Ee_K$ and $\Cc_K$.

%
 

Now suppose that $(\mu;\ba)\in \Cc_K$, and 
  denote by $\Orb(\mu;\ba)$  its orbit under 
 permutations and Cremona transformations.  
   Since   
  $\Orb(\mu;\ba)\subset \Cc_K$, all elements in $\Orb(\mu;\ba)$ are positive. 
  Moreover, if $(\mu;\ba)$ is ordered,  a Cremona move decreases $\mu$ unless 
  $(\mu;\ba)$ is also reduced.  Hence
  we can transform an ordered $(\mu;\ba)$  to a reduced element 
  $(\mu';\ba'): = C(\mu;\ba)$ by a
   sequence  of $k$ moves $C_1,\dots, C_k$
   each consisting of $Cr$  followed by a reordering.
   Thus $C: = C_k\circ \dots\circ C_1$.
   Take any $(d;\bm)\in \Ff^+$ with $(d;\bm)\cdot (d;\bm) < 0$ and 
   denote by $(d';\bm'):= C (d;\bm)$ its image 
   under these moves.  Then we must check that
   $(d';\bm')\cdot (\mu';\ba') \ge 0$.
   
   There are three cases to consider.
   \MS
   
   \NI {\bf Case (i)}   {\it $(d';\bm')$ is positive. }
   
   Since $(d';\bm')\cdot (d';\bm') =(d;\bm)\cdot (d;\bm) < 0$ and $(d';\bm')\in \Ff$ we must have   $-K\cdot (d';\bm')  > 0$.      Hence the result follows from Lemma \ref{le:red}.
 
\MS

 \NI {\bf Case (ii)}   {\it $d'>0$ but some $m_i'<0$. }
 
 In this case, let $(d',\bm'')$ be the positive tuple obtained from $(d';\bm')$ by 
 replacing the negative terms $m_i'$ by $0$.  Because $m^2 + m \ge 0$ for all $m$, 
 we still have $(d';\bm'')\in \Ff^+$.  Further
 $$
  (d';\bm')\cdot (\mu';\ba') \ge 
 (d';\bm'')\cdot (\mu';\ba').
 $$
 Therefore it suffices to show that $(d';\bm'')\cdot (\mu';\ba')\ge 0$.
 If  $(d';\bm'')\cdot (d';\bm'')\ge 0$, then  this holds
 by the argument in the first paragraph of this proof.  
 Otherwise $(d';\bm'')\cdot (d';\bm'')< 0$, and  it holds
as in case (i) above. 
   \MS
   
\NI {\bf Case (iii)}   {\it $d'<0$. }

In this case, we show that $  (d';\bm')\cdot (\mu';\ba') \ge 0$ by induction on $k$, the length of the reducing sequence for $(\mu;\ba)$.
Consider the sequence $(d_\ell;\bm_\ell): = C_{\ell}\circ\dots \circ C_1(d;\bm)$ of elements of $\Ff$ obtained by applying the moves $C_i, i=1,\dots,k,$ to $(d;\bm)$, and let $(\mu_\ell;\ba_\ell)$  be the corresponding elements of $\Cc_K$.
Consider the smallest $\ell$ for which $(d_\ell;\bm_\ell) $ is negative.

Suppose first that $d_\ell < 0$.  Then the entries $s_i: = m_{\ell-1,i}$ of the previous term $\bm_{\ell-1}$ are nonnegative, while, if $t: = d_{\ell-1}$ we have 
$$
0< 2t<s_1+s_2+s_3,\quad \sum_{i=1}^3 s_i^2+s_i \le t^2 + 3t.
$$
Therefore, if $\sum_{i=1}^3  s_i^2= \la t^2$ with $\la > 1$,
we have
$$
2t < \sum_{i=1}^3  s_i \le 3t - (\la-1) t^2,
$$
so that $\la < 1 + 1/t$.  Thus in all cases
$\sum_{i=1}^3  s_i^2 \le t(1+t)$.
But  the minimum of the expression $x^2+y^2+z^2$ subject to the constraints $x,y,z\ge 0, x+y+z=2$ is assumed when $x=y=z$ and is $\frac 43$. Therefore $t(1+t) \ge \frac 43 t^2$, which gives $d_{\ell-1} = t\le 3$.  But there is no integral solution for 
$(d_{\ell-1};\bm_{\ell-1})$ with such a low value 
for $d_{\ell-1}$.  Thus this case does not occur.

Hence the first negative element $(d_\ell;\bm_\ell)$ must have $d_\ell>0$ and some negative entry in $\bm_\ell$.  
But then  define $\bm_\ell''$ as in Case (ii) by replacing all 
negative entries in $\bm_\ell$ by $0$.  We saw there that it suffices to show that 
 $(d_\ell;\bm_\ell'')\cdot (\mu_\ell;\ba_\ell) \ge 0$. 
 If  $(d_\ell;\bm_\ell'')\cdot(d_\ell;\bm_\ell'')\ge 0$ this is automatically true.  
Otherwise, since $\ell\ge 1$, it holds by the inductive hypothesis.

This completes the proof of Lemma \ref{le:alg}.\QED

We are now ready to prove the main result.
We denote by $\la E(a,b)$ the ellipsoid $\{\la x:x\in E(a,b)\}$.  Thus $\la E(a,b) = E(\la^2 a, \la^2 b)$ has  corresponding sequence $\la^2\Nn(a,b)$.
\MS

\NI {\bf Proof of Theorem \ref{thm:hof}.} 
  By standard continuity properties as explained in
\cite[Cor.~1,5]{M}, 
it suffices to prove this when $a,b,c,d$ are rational.
Therefore we will suppose that $c\le d$ are mutually prime integers
and that $(a,b) = \la^2 (e,f)$ where $e\le f$ are also mutually prime integers. 
We need to show that there is
 an embedding 
$\Phi_E: \intt \la E(e,f)\se E(c,d)$ exactly if $\la^2\Nn(e,f)\preccurlyeq \Nn(c,d)$.

By Proposition \ref{prop:M} 
it suffices to consider the corresponding ball embedding $\Phi_B$, and by 
 Proposition \ref{prop:H}, this  exists exactly if 
$$
\Nn\Bigl( \la B\bigl(\bW(e,f)\bigr) \sqcup  B\bigl(\bW(d,d-c)\bigr)\Bigr)\;\preccurlyeq\; \Nn(B(d)) = \Nn(d,d).
$$
Since $\Nn\bigl( \la B\bigl(\bW(e,f)\bigr) \sqcup  B\bigl(\bW(d,d-c)\bigr)\bigr) =
\la^2 \Nn(e,f)\# \Nn(d,d-c)$  this condition is equivalent to
$$
(*)\qquad\la^2 \Nn(e,f)\# \Nn(d,d-c) \;\preccurlyeq\; \Nn(d,d).
$$
Thus the theorem will hold if we show that (\textasteriskcentered) is equivalent to the condition
$\la^2 \Nn(e,f)\preccurlyeq \Nn(c,d)$.

But if $\la^2 \Nn(e,f)\preccurlyeq \Nn(c,d)$ then  we have
$$
\la^2 \Nn(e,f)\# \Nn(d,d-c) \preccurlyeq
\Nn(c,d)\# \Nn(d,d-c)  \preccurlyeq \Nn(d,d)
$$
by Lemma \ref{le:assoc} (ii) and Lemma \ref{le:hof2} (i).
Conversely suppose that (\textasteriskcentered) holds. 
Then $\la^2 \Nn(e,f) \preccurlyeq\Nn(c,d)$ by Corollary \ref{cor:2}.
    Hence the two conditions are equivalent.
\QED

\begin{prop}\labell{prop:many}  The disjoint union of open ellipsoids 
$\sqcup_{i=1}^nE(a_i,b_i)$ embeds symplectically in $E(c,d)$ exactly if
$$
\Nn(a_1,b_1)\#\dots\#\Nn(a_n,b_n)\preccurlyeq \Nn(c,d).
$$ 
\end{prop}  
\begin{proof} Though this case is not considered in \cite{M}, the proof of Proposition \ref{prop:M}  works just as well when the domain is a disjoint union of ellipsoids.  Hence, if $c\le d$
the necessary and sufficient condition for this embedding of unions of ellipsoids to exist is that
$$
\Nn\bigl(\sqcup_i \bW(a_i,b_i)\bigr)\# \Nn\bigl(\bW(d-c,d)\bigr)
\preccurlyeq  \Nn(d,d).
$$
The proof of Corollary \ref{cor:1} adapts to show that 
$$
\Nn\bigl(\sqcup_i \bW(a_i,b_i)\bigr) = 
\Nn(a_1,b_1)\#\dots\#\Nn(a_n,b_n) = :\Cc.
$$
Now use Lemma ~\ref{le:hof2} and Corollary \ref{cor:2} as before. 
\end{proof}

\end{document}